\documentclass[11pt,final]{article}
\usepackage{amsmath,amsthm,amsfonts,amssymb,graphicx,hyperref,enumerate,psfrag}
\usepackage{mathrsfs}
\usepackage{fullpage}

\usepackage[utf8]{inputenc} 

\newtheorem{lemma}{Lemma}

\newtheorem{remark}{Remark}
\newtheorem{proposition}[lemma]{Proposition}
\newtheorem{theorem}[lemma]{Theorem}

\newtheorem{corollary}[lemma]{Corollary}
\newtheorem{question}[lemma]{Question}

\newcommand{\EE}{{\mathbb{E}}}

\newcommand{\PP}{{\mathbb{P}}}

\def\rGs{{\mathscr{G}_\bullet}}
\def\rGss{{\mathscr{G}_{\bullet\bullet}}}

\newcommand{\dZ}{\mathbb {Z}}

\newcommand{\dR}{\mathbb {R}}
\newcommand{\dC}{\mathbb {C}}

\newcommand{\cW}{\mathcal {W}}
\newcommand{\cF}{\mathscr {F}}

\newcommand{\cB}{\mathcal {B}}
\newcommand{\cT}{\mathcal {T}}

\newcommand{\cI}{\mathcal {I}}

\newcommand{\cL}{\mathcal {L}}

\newcommand{\cP}{\mathcal {P}}

\newcommand{\wpi}{\widehat{\pi}}

\newcommand{\rH}{\mathscr{H}}
\newcommand{\rI}{\mathscr{I}}

\newcommand{\dd}{{\rm d}}
\newcommand{\dloc}{{\dd}_{\textsc{loc}}}

\newcommand{\wcL}{{\widehat{\cL}}}
\title{Sparse expanders have negative curvature}
\author{Justin Salez\footnote{CEREMADE, CNRS, UMR 7534, Université Paris-Dauphine, PSL University, 75016 Paris, France}}

\begin{document}
\maketitle

\begin{abstract}
We prove that bounded-degree expanders with non-negative Ollivier-Ricci curvature do not exist, thereby solving a long-standing open problem suggested by  Naor and Milman and publicized by Ollivier (2010). In fact,  this remains true even if we allow for a vanishing proportion of large degrees, large eigenvalues, and negatively-curved edges. To establish this, we work directly at the level of  Benjamini-Schramm limits, and exploit the entropic characterization of the Liouville property   on stationary random  graphs to show that non-negative curvature and spectral expansion are incompatible ``at infinity''.  We then transfer this  result to finite graphs via   local weak convergence. The same approach also applies to the Bacry-Emery curvature condition  CD$(0,\infty)$, thereby settling a recent conjecture of Cushing, Liu and Peyerimhoff (2019). %We believe that this ``local weak limit'' approach to mixing properties of Markov chains will have many other applications.
\end{abstract}
\tableofcontents
\section{Introduction}
\subsection{Non-negative curvature}
The \emph{Ricci curvature} of a manifold is a fundamental concept in Riemannian geometry, see e.g.  \cite{MR3726907}.
In two celebrated works \cite{MR2371483,MR2484937}, Ollivier proposed a notion of curvature based on optimal transport which applies to arbitrary metric spaces, hence in particular to the discrete setting of graphs. Specifically, let $G=(V_G,E_G)$ be a locally finite connected graph. As usual, write $\deg_G(x)$ for the degree of a vertex $x$, and $\dd_G(x,y)$ for the length of a minimal path from $x$ to $y$ in $G$. Let also $P_G\colon V_G\times V_G\to[0,1]$ denote the transition matrix of the lazy simple random walk on $G$, i.e.
\begin{eqnarray*}
P_G(x,y) & := & 
\left\{\begin{array}{ll}
\frac{1}{2\deg_G(x)} & \textrm{if }\{x,y\}\in E_G;\\
\frac{1}{2} & \textrm{if }x=y;\\
0 & \textrm{else}.
\end{array}
\right.
\end{eqnarray*} 
The \emph{Ollivier-Ricci curvature} at an edge $\{x,y\}\in E_G$ is defined as
\begin{eqnarray*}
\kappa_G(x,y) & := & 1-\cW_1\left(P_G(x,\cdot),P_G(y,\cdot)\right),
\end{eqnarray*}
where $\cW_1$ denotes the $L^1-$Wasserstein distance on  $\cP_1(V_G,\dd_G)$, see (\ref{def:cW1}) below.  Note that the computation of $\kappa_G(x,y)$  amounts to solving a finite-dimensional linear optimization problem, and is therefore amenable to standard algorithmic techniques (see \cite{doi:10.1080/10586458.2019.1660740} for a beautiful    interactive curvature calculator).  The {Ollivier-Ricci curvature} of the whole graph is then defined as  
\begin{eqnarray*}
\kappa(G) & := & \inf_{\{x,y\}\in E_G} \kappa_G(x,y).
\end{eqnarray*}

This fundamental geometric quantity measures how distances are contracted, on average, under the action of  $P_G$. When $\kappa(G)\ge 0$, the graph $G$ is called \emph{non-negatively curved}. This is the case, for example, when $G$ is the Cayley graph of an abelian group, as witnessed by the obvious coupling that uses the same random generators for both trajectories. Non-negative curvature is equivalent to the requirement that $P_G$ is a contraction under the Wasserstein metric $\cW_1$, and constitutes the essence of the powerful \emph{path coupling method} for bounding mixing times \cite{MR2316551}. Consequences in terms of geometry, mixing,  and concentration of measure  have been massively investigated, and quantified by a variety of functional inequalities. The literature is too vast for an exhaustive account, and we refer the reader to the seminal papers \cite{MR2371483,MR2484937,MR2872958,MR2683634}, the survey \cite{MR2648269}, and the more recent works \cite{MR3726607,MR3998765,MR4096132,MR4149342,2019arXiv190713514M} for details, variations, references, and open problems. In particular, the present work was motivated by the following long-standing question, due to Naor and Milman, and publicized by Ollivier \cite[Problem T]{MR2648269}. 
Recall that a \emph{family of expanders} is a sequence  of finite graphs with uniformly bounded degrees, diverging sizes, and spectral gap bounded away from $0$.

\begin{question}[Problem T in \cite{MR2648269}]\label{quest}Is there a family of non-negatively curved expanders ? 
\end{question}

An instructive special class of graphs for which non-negative curvature is completely understood is that of cubic graphs. Specifically, it was shown in \cite{doi:10.1080/10586458.2019.1660740} that prism graphs and Möbius ladders are the only cubic graphs with non-negative Ollivier-Ricci curvature. Since these  are not expanders, the answer to Question \ref{quest} is negative for cubic graphs. To the best of our knowledge, this is the only result in the direction of Question \ref{quest}, despite the rich body of works on non-negative  curvature.

 \subsection{Main result}
In the present paper, we answer Question \ref{quest} negatively in full generality, as well as its {CD$(0,\infty)$} analogue raised by Cushing, Liu and Peyerimhoff \cite[Conjecture 9.11]{MR4045968}, see Remark \ref{rk:CD} below. Moreover, we show that the answer to  Question \ref{quest} remains negative even if we significantly relax the required properties.
Specifically, denote by $\Delta(G)$ the maximum degree of a finite graph $G$, and by
$$1\ =\ \lambda_1(G)\ \ge\ \lambda_2(G)\ \ge \ldots\ \ge \ \lambda_{N}(G)\ \ge\  0,$$
the $N=|V_G|$ ordered eigenvalues of its transition matrix $P_G$. With these notations,  Question \ref{quest} simply asks whether there exist constants $\Delta\ge 1$, $\rho<1$ and arbitrary large graphs   satisfying  
\begin{enumerate}[(A)] 
\item sparsity: $\Delta(G)  \le  \Delta;$
\item spectral expansion: $\lambda_2(G)  \le  \rho;$
\item non-negative curvature: $\kappa(G) \ge  0.$
\end{enumerate}
Our main result says that no large graph can  even come close to satisfying these three requirements. 

\begin{theorem}[Main result]\label{th:main}Fix  $\Delta\ge 1$ and $\rho\in(0,1)$. 
Then, there exists a constant $\varepsilon=\varepsilon_{\Delta,\rho}>0$ such that \emph{every} finite graph $G$ must satisfy one of the following  conditions:
\begin{itemize}
\item either $G$ is far from satisfying the sparsity requirement (A), in the following sense:
\begin{eqnarray*}
\sum_{x\in V_G}\deg_G(x)\log \deg_G(x)  & > & (\Delta\log \Delta) |V_G| ;
\end{eqnarray*}
\item or $G$ is far from satisfying the expansion requirement (B), in the following sense:
\begin{eqnarray*}
\mathrm{card}\{i\colon \lambda_i(G)>\rho\} & \ge & \varepsilon|V_G| ;
\end{eqnarray*}
\item or $G$ is far from satisfying the curvature requirement (C), in the following sense:
\begin{eqnarray*}
\mathrm{card}\{e\in E_G\colon \kappa_G(e)<-\varepsilon\} & \ge  & \varepsilon|E_G|.
\end{eqnarray*}
\end{itemize}
\end{theorem}

Note that the conclusion is only meaningful for large graphs, since the second condition  is trivially satisfied when $|V_G|\le \frac{1}{\varepsilon}$. Here is an equivalent -- but perhaps more intuitive -- formulation.

\begin{theorem}[Rephrasing]\label{th:reform} Let $G_n=(V_n,E_n), n\ge 1$ be finite graphs with the sparsity property
\begin{eqnarray}
\label{assume:sparse}
\sup_{n\ge 1}\left\{\frac{1}{|V_n|}\sum_{x\in V_n}\deg_{G_n}(x)\log\deg_{G_n}(x)\right\} & < & \infty.
\end{eqnarray}
Suppose in addition that the Ollivier-Ricci curvature is almost non-negative on most edges, i.e.
\begin{eqnarray}
\forall\varepsilon>0,\quad \frac{1}{|E_n|}\,\mathrm{card}\{e\in E_n\colon \kappa_{G_n}(e)<- \varepsilon\} & \xrightarrow[n\to\infty]{}  & 0.
\end{eqnarray}
Then, a macroscopic proportion of eigenvalues of the transition matrix must accumulate near $1$:
\begin{eqnarray}
\label{poorexpansion}
\forall\rho<1,\quad \liminf_{n\to\infty}\left\{\frac{1}{|V_n|}\,\mathrm{card}\{i\colon \lambda_i(G_n)\ge \rho\}\right\} &  > &  0.
\end{eqnarray}
\end{theorem}
Here again, the theorem is only meaningful in the large-size limit $|V_n|\to \infty$, since  the conclusion  (\ref{poorexpansion}) trivially holds otherwise. The high-level message is that on large sparse graphs, non-negative curvature (in an even weak sense) induces extremely poor spectral expansion.  This  stands in stark contrast with the  traditional idea -- quantified by a broad variety of functional inequalities over the past decade -- that non-negative curvature is associated with  \emph{good} mixing behavior.

\begin{remark}[Bacry-Emery curvature]\label{rk:CD} Bacry and Emery \cite{MR889476,MR941980,MR3155209} developed a different notion of non-negative curvature based on $\Gamma-$calculus and known as the \emph{CD$(0,\infty)$} condition, see also  \cite{MR3492631,MR3706773}. Since this notion is local, our proof also applies, with the role of Theorem \ref{th:curvature} being played by a recent result of Hua \cite[Theorem 2]{MR3910593}. Consequently, there is no family of expanders satisfying \emph{CD$(0,\infty)$}, as conjectured by Cushing, Liu and Peyerimhoff \cite[Conjecture 9.11]{MR4045968}. We note that the weaker statement obtained by replacing \emph{CD$(0,\infty)$} with \emph{CD$(0,n)$} was recently established by Münch \cite{2019arXiv190910242M}. We warmly thank David Cushing, Shiping Liu and Florentin Münch for pointing this out. 
\end{remark}

 \begin{remark}[Laziness]The literature actually contains a whole family of variants $(\kappa_\alpha)_{\alpha\in[0,1)}$  of the Ollivier-Ricci curvature $\kappa$, obtained by replacing the matrix $P_G$ with its $\alpha-$idle version:
\begin{eqnarray*}
P_{G}^{(\alpha)} & := & (2-2\alpha) P_G+(2\alpha-1)\,{\rm Id}.
\end{eqnarray*} 
There is even a continuous-time version $\kappa_{\star} :=  \lim_{\alpha\to 1} \frac{\kappa_\alpha}{1-\alpha}$, proposed in \cite{MR2872958} and largely adopted since then. In fact, it was later shown (see  \cite[Remark 5.4]{MR3815539})  that   
$
\frac{\kappa_\alpha}{1-\alpha}  \le  \kappa_\star \ = \ 2\kappa,
$
where $\kappa=\kappa_{1/2}$ is the version considered in the present paper.  
Consequently, our result is stated in the strongest possible form, and applies to all  versions of the Ollivier-Ricci curvature.
\end{remark}

\begin{remark}[Eigenvectors]\label{rk:eigenvector}Our proof will actually reveal more than (\ref{poorexpansion}): not only are there many eigenvalues near $1$, but the corresponding eigenvectors furthermore charge most vertices significantly. In other words, the poor spectral expansion of non-negatively curved graphs is not restricted to any specific region: it applies everywhere. See Remark \ref{rk:final} for a precise statement.
\end{remark}

\subsection{Proof outline}
\paragraph{Proof outline.} The most natural route towards Question \ref{quest}  would consist in looking for a quantitative upper-bound on the spectral gap of a finite non-negatively curved graph, in terms of its size and maximum degree. Interestingly, we do \emph{not} pursue this approach here. Neither do we try to obtain asymptotic estimates along a sequence of sparse graphs $(G_n)_{n\ge 1}$ with non-negative curvature. Instead, we work directly at the elegant level of \emph{local weak limits} of finite graphs, and exploit their built-in \emph{stationarity}  to prove that non-negative curvature and spectral expansion are incompatible ``at infinity''. This relies on the central concept of \emph{asymptotic entropy},  and  its classical relations  with the Liouville property and the spectral radius. We then transfer this incompatibility result to finite graphs via a relative-compactness argument. As far as we know, the idea of using local weak limits as a tool to deduce generic  bounds on the mixing parameters of sparse Markov chains have not received much attention. We firmly believe that this viewpoint will have many applications. 

\paragraph{Further questions.} The surprising ``$\deg\log \deg$'' requirement (\ref{assume:sparse}) is used  to define the asymptotic entropy on which our whole argument relies. We do not know whether it is necessary for the conclusion (\ref{poorexpansion}) to hold, or whether it can be further relaxed. Note that some degree restriction is necessary, since the complete graph satisfies $\lambda_2(G)=\kappa(G)=1/2$, regardless of its size. Also, a drawback of our approach -- as of any limit argument -- is its non-quantitative nature. It would be interesting to find an explicit  upper-bound (vanishing as $n\to\infty$) on the spectral gap of a non-negatively curved graph with $n$ vertices and maximum degree $\Delta$, i.e. to estimate
\begin{eqnarray*}
\gamma_\Delta(n) & := & \max\{1-\lambda_2(G)\colon |V_G|=n,\Delta(G)\le\Delta,\kappa(G)\ge 0\}.
\end{eqnarray*}

\paragraph{Organization of the paper.} The remainder of the paper is organized as follows:   Section \ref{sec:lwc} offers a brief, self-contained introduction to the framework of random rooted graphs. In particular, we recall the definition of  local weak convergence (Section \ref{sec:def}),  introduce  the  key notions of \emph{unimodularity}, \emph{stationarity} and \emph{tightness} (Section \ref{sec:prop}), and gather important results on the \emph{asymptotic entropy} of random walks on stationary graphs (Section \ref{sec:entropy}). Section \ref{sec:proof} is devoted to the proof of the main result, which is  reduced (in Section \ref{sec:stage}) to the following two main steps: 
\begin{enumerate}
\item Proving that  non-negative curvature implies zero-entropy  (Section \ref{sec:curvature}).
\item Proving that zero-entropy causes poor  spectral expansion (Section \ref{sec:spectrum}).
\end{enumerate}

\paragraph{Acknowledgment.} The author warmly thanks Itai Benjamini, David Cushing, Nicolas Curien, Shiping Liu, Russell Lyons, Florentin Münch and Pierre Pansu for many wonderful comments, connections and references. This work was partially supported by Institut Universitaire de France.

\section{Random rooted graphs}
\label{sec:lwc}
In this section, we provide a self-contained introduction to  the framework of \emph{local weak convergence}. This limit theory for sparse graphs was introduced by Benjamini and Schramm \cite{MR1873300} and developed further by Aldous and Steele \cite{MR2023650} and Aldous and Lyons \cite{MR2354165}. The limit points   are \emph{random rooted graphs} enjoying a powerful form of \emph{stationarity}. They describe the ``internal'' geometry of large graphs, as seen from a uniformly chosen vertex. Local weak limits  are often much more convenient to work with than the finite-graph sequences that they approximate, and have been shown to capture the asymptotic behavior of a number of natural graph parameters, see, e.g. \cite{MR2160416,MR3101844,MR2789584,MR3449319}. The present paper can be viewed as another illustration of the strength of this modern viewpoint. 

\subsection{Local weak convergence}\label{sec:def}
\paragraph{The space of rooted graphs.} 
All graphs considered in this paper will be simple, undirected, countable, and locally finite. A \emph{rooted graph} is a pair $(G,o)$, where $G$ is a  graph and $o$ is a distinguished vertex, called the \emph{root}.  Two rooted graphs $(G,o)$ and $(G',o')$ are \emph{isomorphic}, written $G\simeq G'$, if there is a bijection $\phi\colon V_G\to V_{G'}$ which preserves the root ($\phi(o)=o'$) and the edges:
\begin{eqnarray*}
\forall x,y\in V_G,\quad \{x,y\}\in E_G & \Longleftrightarrow & \left\{\phi(x),\phi(y)\right\}\in E_{G'}.
\end{eqnarray*}
We let $\rGs$ denote the set of connected rooted graphs, considered up to the isomorphism relation $\simeq$. To lighten the exposition, we will use the same notation $(G,o)$ for the rooted graph and its equivalence class. We write $\cB_t(G,o)$ for the \emph{ball of radius $t$ around the root} in $G$, i.e. the (finite) rooted subgraph of $G$ induced by the set $\{x\in V_G\colon \dd_G(o,x)\le t\}$. We equip  $\rGs$ with the \emph{local metric} $\dloc\colon\rGs\times \rGs\to [0,1]$, defined by 
\begin{eqnarray*}
\dloc((G,o),(G',o')) & := & \frac{1}{1+r},\quad \textrm{ with } \quad r\ =\ \sup\{t\ge 0\colon \cB_t(G,o)\simeq\cB_{t}(G',o')\}.
\end{eqnarray*}
In words, two elements of $\rGs$ are ``close'' to each other if one has to look  ``far away'' from the root to distinguish them apart.   It can be shown that $(\rGs,\dloc)$ is a complete  separable metric space. We equip it with its Borel $\sigma-$algebra, and call $\rGs-$valued random variables \emph{random rooted graphs}.

\paragraph{Local weak convergence.} Write $\cP(\rGs)$ for the space of Borel probability measures on $\rGs$, equipped with the usual topology of weak convergence. If $G$ is an arbitrary finite graph, define its \emph{local profile} $\cL_G\in\cP(\rGs)$ to be the empirical distribution of all possible \emph{rootings} of $G$, i.e.
\begin{eqnarray}
\label{def:profile}
\cL_G & := & \frac{1}{|V_G|}\sum_{x\in V_G}\delta_{(G,x)},
 \end{eqnarray} 
where $(G,x)$ is here implicitly restricted to the connected component of $x$ if $G$ is not connected. Finally, if $G_n=(V_n,E_n), {n\ge 1}$ are finite graphs whose local profiles $(\cL_{G_n})_{n\ge 1}$ admit a limit $\cL$ in $\cP(\rGs)$, we call $\cL$ the \emph{local weak limit} of the sequence $(G_n)_{n\ge 1}$, and write simply 
\begin{eqnarray*}
G_n & \xrightarrow[n\to\infty]{} & \cL.
\end{eqnarray*}
In words, $\cL$ is the law of a random rooted graph   which describes how the deterministic graph $G_n$ asymptotically looks when seen from a uniformly chosen root. 
More formally,  
\begin{eqnarray}
\label{def:lwc}
\frac{1}{|V_n|}\sum_{x\in V_n} f(G_n,x) & \xrightarrow[n\to\infty]{} & \cL\left[f(G,o)\right] \ \triangleq \ \int_\rGs f\,\dd\cL,
\end{eqnarray}
for each continuous, bounded observable $f\colon\rGs\to\dR$.  The left-hand side  can be thought of as a spatial average of ``local contributions'' from the various vertices of $G_n$. In short, local weak convergence  allows one to conveniently replace the asymptotic analysis of such  averages with the direct computation of an expectation at the root of a certain random graph.

\paragraph{Local observables.} The class of continuous functions on $\rGs$ clearly contains (but is not restricted to) all $t-$\emph{local} observables $(t\ge 0)$, where  $f\colon\rGs\to\dR$ is called $t-$\emph{local} if the value $f(G,o)$ is determined by the (isomorphic class of the) finite ball $\cB_t(G,o)$. Here is a short list of  examples, which will be used throughout the paper without notice:
\begin{itemize}
\item The  root degree  $(G,o)\mapsto \deg_G(o)$ is $1-$local. 
\item The minimum curvature at $o$, $(G,o)\mapsto \min_{x\sim o}\kappa_G(o,x)$ is $2-$local.
\item For each $t\ge 0$, the return probability $(G,o)\mapsto P_G^t(o,o)$ is $t-$local (in fact, $(\lfloor t/2\rfloor +1)-$local).
\item For each $t\ge 0$, the  $t-$step entropy $(G,o)\mapsto -\sum_{x\in V_G}P_G^t(o,x)\log P_G^t(o,x)$ is $t-$local.
\end{itemize}

\subsection{Tightness, unimodularity and stationarity}
\label{sec:prop}
\paragraph{Tightness.} One of the many reasons for the success of the local weak convergence framework (compared to other limit theories for sparse graphs)  is the fact that every ``reasonable'' sequence of sparse graphs admits a local weak limit. The following tightness criterion, due to Benjamini, Lyons and Schramm, gives an honest mathematical content to this vague claim. Note, of course, that passing to sub-sequences is unavoidable.
\begin{theorem}[Tightness, see Theorem 3.1 in \cite{MR3316916}]\label{th:tightness}Let $G_n=(V_n,E_n),{n\ge 1}$ be  finite graphs so that
\begin{eqnarray*}
\sup_{n\ge 1}\left\{\frac{1}{|V_n|}\sum_{x\in V_n}\phi\left(\deg_{G_n}(x)\right)\right\} & < & \infty,
\end{eqnarray*}
for some function $\phi\colon\dZ_+\to\dR_+$ satisfying $\phi(d)  \gg  d$  as $d\to\infty$. Then,  $(G_n)_{n\ge 1}$ has a subsequence which admits a local weak limit.
\end{theorem}
In particular, this criterion applies to the sequence $(G_n)_{n\ge 1}$ in Theorem \ref{th:reform}, with $\phi(d)=d\log d$. This will ensure that we can ``pass to the limit'' and study the question of existence of non-negatively curved expanders directly at the level of local weak limits. 

\paragraph{Unimodularity.} Local weak limits of finite graphs happen to enjoy a powerful  distributional invariance, which is directly inherited from the fact that the root is equally likely to be any vertex under the local profile (\ref{def:profile}). More precisely, a measure $\cL\in\cP(\rGs)$ is  called \emph{unimodular} if it satisfies
\begin{eqnarray}
\label{def:mtp}
\cL\left[\sum_{x\in V_G}f(G,o,x)\right]  & = & \cL \left[\sum_{x\in V_G}f(G,x,o)\right],
\end{eqnarray}
for every Borel function $f\colon \rGss\to [0,\infty]$, where $\rGss$ denotes the analogue of the space $\rGs$ with two distinguished roots instead of one.  Thinking of $f(G,o,x)$ as an amount of mass sent from  $o$ to $x$, the identity (\ref{def:mtp}) expresses the fact that the expected masses received and sent by the root coincide. This  \emph{Mass Transport Principle}  is clearly satisfied when $\cL$ is the local profile of a finite graph, and is  preserved under weak convergence. Thus, we obtain the following   fundamental result.
\begin{theorem}[Inherited unimodularity]\label{lm:uni}
All local weak limits of finite graphs are unimodular.
\end{theorem}
 Whether the converse holds  is a notoriously hard open problem with deep implications, see  \cite{MR2354165,MR2776719,MR3316916}.
Let us here record a first simple consequence of unimodularity, which will be useful.
\begin{lemma}[Everything shows at the root, see Lemma 2.3 in \cite{MR2354165}]\label{lm:root}Suppose that $\cL\in\cP(\rGs)$ is unimodular, and let $B\subseteq\rGs$ be a Borel set such that $\cL(B)=1$. Then we also have,
\begin{eqnarray*}
\cL\left(\{\forall x\in V_G,\ (G,x)\in B\}\right) & = & 1.
\end{eqnarray*}
\end{lemma}
\begin{proof}Just apply the Mass Transport Principle with $f(G,o,x)={\bf 1}_{(G,o)\notin B}$.
\end{proof}

\paragraph{Stationarity.} Under a mild integrability condition and a trivial change of measure, unimodularity can be rephrased as \emph{reversibility} under a natural Markov chain on $\rGs$. We will here only need the weaker notion of \emph{stationarity}. Specifically, we say that a law $\cL\in\cP(\rGs)$  is \emph{stationary} if it is invariant for the Markov chain on $\rGs$ which, at each step, keeps the underlying graph as it is and moves the root according to the transition matrix $P_G$. In other words, $\cL$ is stationary if
\begin{eqnarray}
\label{def:statio}  \cL\left[\sum_{x\in V_G}P^t_G(o,x)h(G,x)\right] & = & \cL\left[h(G,o)\right],
\end{eqnarray}
for every Borel function $h\colon\rGs\to[0,\infty]$ and every $t\ge 0$ (equivalently, for $t=1$). The relation with unimodularity is summed up in the following classical lemma (see, e.g. \cite{MR2994841}).

\begin{lemma}[Degree-biasing] \label{lm:unimod}Let $\cL\in\cP(\rGs)$ be a unimodular law with
$
{\deg}(\cL) :=  \cL[\deg_G(o)] < \infty.
$
Then, the law $\widehat{\cL}\in\cP(\rGs)$ defined by the following change of measure is stationary:
\begin{eqnarray}
\label{def:biased}
\dd \wcL (G,o) & := & \frac{\deg_G(o)}{\deg(\cL)}\,\dd\cL(G,o).
\end{eqnarray}
\end{lemma}
\begin{proof}
Apply the Mass Transport Principle  to $\cL$ with $f(G,o,x)=h(G,o){\bf 1}_{\{x,o\}\in E_G}$.
\end{proof}
\begin{remark}[Mutual absolute continuity]\label{rk:absolute}
It follows  from (\ref{def:biased}) that the original  law $\cL$ and its degree-biased version $\wcL$ are mutually absolutely continuous. In other words, we have
\begin{eqnarray*}
\cL(B)=1 & \Longleftrightarrow & \wcL(B)=1,
\end{eqnarray*}
for any Borel set $B\subseteq \rGs$, allowing us to transfer results from one law to the other. 
\end{remark}
\subsection{Spectral radius, entropy and the Liouville property}
\label{sec:entropy}
Stationarity is a powerful property, because it enables the development of an \emph{ergodic theory} of random rooted graphs. See the inspiring works \cite{MR1336708} on Galton-Watson trees, \cite{MR2994841} on random rooted graphs, and \cite{MR3395463} on general random environments. In particular, a classical application of  Kingman's sub-additive ergodic theorem allows one to define the (quenched)   \emph{asymptotic entropy} of random walks on stationary random  graphs, as recalled in the following lemma. 
\begin{lemma}[Entropy]\label{lm:entropy}Let $\cL\in\cP(\rGs)$ be stationary  with $\cL[\log\deg_G(o)]<\infty$. Then the limit
\begin{eqnarray*}
\rH(G,o) & := & \lim_{t\to\infty}\frac{1}{t}\sum_{x\in V_G}P^t_G(o,x)\log\frac{1}{P_G^t(o,x)},
\end{eqnarray*}
exists $\cL-$almost-surely and in $L^1(\rGs,\cL)$, and does not depend on the choice of the root $o$.  
\end{lemma}
We will henceforth simply write  $\rH(G)$ instead of $\rH(G,o)$, and call this the \emph{entropy} of $G$.
\begin{proof} 
Let $(G,o)$ have law $\cL$, and conditionally on  $(G,o)$, let $X=(X_t)_{t\ge 0}$ be a lazy simple random walk on $G$ starting from $X_0=o$. For  $0\le s\le t$, define a non-negative random variable $Z_{s,t}$ by 
\begin{eqnarray*}
Z_{s,t} & := & \log\frac{1}{P_G^{t-s}(X_s,X_t)}.
\end{eqnarray*}
Note that $Z_{t,s}\stackrel{d}{=}Z_{0,t-s}$. Indeed, for any Borel function $f\colon \dR_+\to\dR_+$, we have by definition
\begin{eqnarray*}
\EE\left[f(Z_{s,t})\right] & = & \EE\left[\sum_{x,y\in V_G}P_G^{s}(o,x)P_G^{t-s}(x,y)f\left(\log\frac{1}{P_G^{t-s}(x,y)}\right)\right]\\
& = & \EE\left[\sum_{y\in V_G}P_G^{t-s}(o,y)f\left(\log\frac{1}{P_G^{t-s}(o,y)}\right)\right]\\
& = & \EE\left[f(Z_{0,t-s})\right],
\end{eqnarray*}
where the second line uses the stationarity (\ref{def:statio}) with $h(G,o)=\sum_{y}P_G^{t-s}(o,y)f\left(\log\frac{1}{P_G^{t-s}(o,y)}\right)$.
Moreover, the trivial inequality $P_G^{t}(o,y)\ge P_G^s(o,x)P_G^{t-s}(x,y)$ readily  implies the sub-additive property
\begin{eqnarray}
\label{sub:add}
Z_{0,t} & \le & Z_{0,s}+Z_{s,t}.
\end{eqnarray}
Finally, the  assumption $\cL[\log\deg_G(o)]<\infty$ ensures that $\EE[Z_{0,1}]<\infty$. Consequently,  Kingman's sub-additive ergodic theorem (see, e.g. \cite[Theorem 14.44]{MR3616205}) guarantees the existence of a non-negative, integrable random variable $Z_\infty$ such that almost-surely and in $L^1$,
\begin{eqnarray*}
\frac{Z_{0,t}}{t} & \xrightarrow[t\to\infty]{} & Z_\infty.
\end{eqnarray*}
Averaging this convergence over the random walk $X$ (i.e., taking conditional expectation given the random rooted graph) yields the existence of the limit $\rH(G,o)$. By Lemma \ref{lm:root}, the same is true if $o$ is replaced by any $x\in V_G$. Moreover, the sub-additive property (\ref{sub:add}) with $s=1$ shows that 
\begin{eqnarray*}
\rH(G,o) & \le & \sum_{x\in V_G}P_G(o,x)\rH(G,x),
\end{eqnarray*}
$\cL-$almost-surely. Since $\theta\mapsto (\theta-a)_+$ is monotone and convex for $a\ge 0$, this inequality implies  
\begin{eqnarray*}
\forall a\ge 0,\quad \left(\rH(G,o)-a\right)_+ & \le & \sum_{x\in V_G}P_G(o,x)\left(\rH(G,x)-a\right)_+.
\end{eqnarray*}
 But  the two sides have the same law  by stationarity, so they must coincide $\cL-$almost-surely. The fact that this is true for all $a\ge 0$   deterministically forces the equality
$ \rH(G,x) =\rH(G,o)$ for all neighbours $x$ of $o$, and hence for all $x\in V_G$ by Lemma \ref{lm:root}. \end{proof}
\paragraph{The Liouville property.} One of the interests of asymptotic entropy   lies in its relation with the Liouville property.  A function $f\colon V_G\to \dR$ is called \emph{harmonic} on $G$ if $P_Gf=f$, where
\begin{eqnarray}
\label{def:P}
\forall x\in V_G,\quad (P_Gf)(x) & := & \sum_{y\in V_G}P_G(x,y)f(y).
\end{eqnarray}
This is trivially the case, in particular, when $f$ is constant. The graph $G$ has the \emph{Liouville property} if it admits no non-constant bounded harmonic function. For stationary random graphs, this functional-analytic property turns out to admit the following simple entropic characterization. 
\begin{theorem}[Entropic characterization of the Liouville property]\label{th:entropy}The  equivalence 
\begin{eqnarray*}
\rH(G)=0  & \Longleftrightarrow & G \textrm{ has the Liouville property},
\end{eqnarray*}
holds almost-surely under any stationary law $\cL\in\cP(\rGs)$  with $\cL[\log\deg_G(o)]<\infty$. 
\end{theorem}
This remarkable result has a long history: it originates with the pioneering works of Avez  \cite{MR324741,MR353405,MR0507229}, and was then made famous in a celebrated paper of Kaimanovich and Vershik \cite{MR704539}. In the present setting of stationary random graphs, the  implication $\Longrightarrow$ was established by Benjamini and Curien \cite{MR2994841}, and refined by Benjamini, Duminil-Copin, Kozma and Yadin  \cite{MR3395463}. The converse $\Longleftarrow$  was proved by Carrasco Piaggio and Lessa \cite{MR3546392} (see also \cite{MR3813982}), but under an additional growth assumption. Since this is the implication that we are going to use, we need to  give more details. 

\begin{proof}[Proof of Theorem \ref{th:entropy}]Fix a connected graph $G$, and let $X=(X_t)_{t\ge 0}$ denote a lazy simple random walk on $G$ starting at some fixed vertex $o\in V_G$. Write ${\bf P}^G$ for its law, which is a probability measure on the product space $V^{\dZ_+}_G$. On this space, let $\cI$ denote the $\sigma-$field of all events which are invariant under the natural shift  $(x_t)_{t\ge 0}\mapsto(x_{t+1})_{t\ge 0}$. Then \cite[Proposition 14.12]{MR3616205} states that
\begin{eqnarray*}
G\textrm{ has the Liouville property} & \Longleftrightarrow & \cI\textrm{ is ${\bf P}^G-$trivial}.
\end{eqnarray*}
On the other hand, writing $\cT=\bigcap_{t=0}^\infty\sigma(x_t,x_{t+1},\ldots)$ for the tail $\sigma-$field on $V^{\dZ_+}_G$, we have
\begin{eqnarray*}
\cI\textrm{ is ${\bf P}^G-$trivial} & \Longleftrightarrow & \cT\textrm{ is ${\bf P}^G-$trivial},
\end{eqnarray*}
by Theorem \cite[Theorem 14.18]{MR3616205} and because $X$ is lazy. Finally, the equivalence
\begin{eqnarray*}
\cL\left(\cT\textrm{ is ${\bf P}^G-$trivial}\right)=1 & \Longleftrightarrow & \cL(\rH(G)=0)=1,
\end{eqnarray*}
was proved in  \cite[Theorem 3.2]{MR2994841} for any stationary law $\cL$ with $\cL[\log\deg_G(o)]<\infty$. Thus,
\begin{eqnarray}
\label{annealed}
\cL(G\textrm{ has the Liouville property})=1 & \Longleftrightarrow & \cL\left(\rH(G)=0\right)=1,
\end{eqnarray}
and this annealed statement will actually suffice for  the present paper. However, deducing the  quenched claim is easy, as we now explain. Define the events  $A:=\{G\textrm{ has the Liouville property}\}$ and $B:=\{\rH(G)=0\}$, and let $A\Delta B$ denote their symmetric difference. We want to show that
\begin{eqnarray}
\label{quenched}
\cL(A\Delta B)& = & 0,
\end{eqnarray}
for any stationary law $\cL$ with $\cL[\log\deg_G(o)]<\infty$. We already know this  if $A,B$ are $\cL-$trivial, thanks to (\ref{annealed}). Moreover, the events $A,B$ are clearly \emph{root-invariant}, in the sense that 
\begin{eqnarray*}
(G,o)\in A & \Longrightarrow & \{\forall x\in V_G,(G,x)\in A\}.
\end{eqnarray*}
Consequently, (\ref{quenched}) holds under the extra assumption that \emph{root-invariant events are $\cL-$trivial}. But this is known as \emph{ergodicity}, and any stationary law can be decomposed as a mixture of ergodic laws, by \cite[Theorem 4.7]{MR2354165}. Thus, (\ref{quenched}) extends to all stationary laws $\cL$ with $\cL[\log\deg_G(o)]<\infty$.
\end{proof}

\paragraph{Spectral radius.} The entropy $\rH(G)$ is related to several other fundamental graph-theoretical quantities, such as the \emph{speed},  \emph{growth}, or  \emph{spectral radius}, see \cite{MR3616205}. Let us recall the last notion. Fix a rooted graph $(G,o)\in\rGs$. For any $t,s\ge 0$, we trivially have
$
P_G^{t+s}(o,o)  \ge P_G^t(o,o)P_G^s(o,o).
$
By Fekete's lemma, we deduce that the limit
\begin{eqnarray}
\label{def:radius}
\varrho(G,o) & := & \lim_{t\to\infty}\left(P_G^t(o,o)\right)^{\frac 1 t},
\end{eqnarray}
exists in $(0,1]$. Moreover, the connectivity of $G$ together with the trivial inequality 
\begin{eqnarray*}
P^{t+2s}_G(o,o) & \ge & P^s_G(o,x)P^t_G(x,x)P^s_G(x,o),
\end{eqnarray*} shows that $\varrho(G,o)$ does not depend on the choice of the root $o$. Thus, we will henceforth simply write  $\varrho(G)$, and call this quantity the \emph{spectral radius} of $G$.
\begin{lemma}[Spectral radius vs entropy]\label{lm:radius}The inequality
\begin{eqnarray*}
 \rH(G) & \ge & 2\log\frac{1}{\varrho(G)},
\end{eqnarray*}
holds almost-surely under any stationary law $\cL$ with $ \cL[\log\deg_G(o)]<\infty$.
\end{lemma}
\begin{proof}For any rooted graph $(G,o)$ and any $t\ge 0$, we have by concavity 
\begin{eqnarray*}
\log\left(P_G^{2t}(o,o)\right) & = & \log\left(\sum_{x\in V_G}P_G^t(o,x)P^t_G(x,o)\right)\\
& \ge &   \sum_{x\in V_G}P_G^t(o,x)\log P_G^t(x,o)\\
& = & \sum_{x\in V_G}P_G^t(o,x)\log P_G^t(o,x)+\sum_{x\in V_G}P_G^t(o,x)\log\left(\frac{\deg_G(o)}{\deg_G(x)}\right),
\end{eqnarray*}
where the last line uses the reversibility $\deg_G(o)P_G^t(o,x)=\deg_G(x)P_G^t(x,o)$.
Dividing by $-2t$ and taking the limit as $t\to\infty$ in $L^1(\rGs,\cL)$  yields the claim, provided we can show that 
\begin{eqnarray*}
\frac{1}{t} \sum_{x\in V_G}P_G^t(o,x)\log\left(\frac{\deg_G(o)}{\deg_G(x)}\right) & \xrightarrow[t\to\infty]{L^1(\rGs,\cL)} & 0.
\end{eqnarray*}
But this follows from the crude bound
\begin{eqnarray*}
 \cL\left[\left|\sum_{x\in V_G}P_G^t(o,x)\log\left(\frac{\deg_G(o)}{\deg_G(x)}\right) \right|\right] & \le &  \cL\left[\sum_{x\in V_G}P_G^t(o,x)\left(\log \deg_G(o)+\log\deg_G(x)\right)\right]\\
& = & 2 \cL\left[\log \deg_G(o)\right],
\end{eqnarray*}
where the second line simply uses the stationarity property (\ref{def:statio}) with $h(G,o)=\log\deg_G(o)$. 
\end{proof}

\begin{remark}[Unimodular analogues] \label{rk:uni} By Lemma \ref{lm:unimod} and Remark \ref{rk:absolute}, all results in this section also apply to any unimodular law  $\cL\in\cP(\rGs)$  with $ \cL[\deg_G(o)\log\deg_G(o)]<\infty$.
\end{remark}

\section{Proof of the main result}
\label{sec:proof}
We are now ready to prove our main result. We work with  the formulation given in Theorem \ref{th:reform}.  Section \ref{sec:stage} below reduces it to two key results, which are then proved in Sections \ref{sec:curvature} and \ref{sec:spectrum}.
\subsection{Setting the stage}
\label{sec:stage}
Let $G_n=(V_n,E_n)$, $n\ge 1$ be finite graphs satisfying the assumptions of Theorem \ref{th:reform}, i.e.
\begin{eqnarray}
\label{assume:sparse2}
\sup_{n\ge 1}\left\{\frac{1}{|V_n|}\sum_{x\in V_n}\deg_{G_n}(x)\log\deg_{G_n}(x)\right\} & < & \infty;\\
\label{assume:nonneg2}
\forall\varepsilon>0,\quad \frac{1}{|E_n|}\,\mathrm{card}\{e\in E_n\colon \kappa_{G_n}(e)<- \varepsilon\} & \xrightarrow[n\to\infty]{}  & 0.
\end{eqnarray}
 Recall that our goal  is to establish 
\begin{eqnarray}
\label{poorexpansion2}
\forall\rho\in(0,1),\quad \liminf_{n\to\infty}\left\{\frac{1}{|V_n|}\,\mathrm{card}\{i\colon \lambda_i(G_n)>\rho\}\right\} &  > &  0.
\end{eqnarray}
By (\ref{assume:sparse2}) and Theorem \ref{th:tightness}, we may assume, upon extracting a subsequence if necessary, that 
\begin{eqnarray}
\label{lwc}
G_n & \xrightarrow[n\to\infty]{} & \cL,
\end{eqnarray}
for some $\cL\in\cP(\rGs)$. Note that $\cL$ is automatically unimodular by Theorem \ref{lm:uni}, and such that
\begin{eqnarray}
\cL\left[\deg_G(o)\log\deg_G(o)\right] & < & \infty.
\end{eqnarray}
Just like the degree, the  curvature is a local notion, hence it also ``passes to the limit'', i.e
\begin{eqnarray}
\label{kappaG}
\cL\left(\kappa(G)\ge 0\right) & = & 1.
\end{eqnarray}
\begin{proof}As already mentioned, the observable $f\colon (G,o)\mapsto\min_{x\sim o}\kappa_G(o,x)$  is $2-$local, hence continuous on $\rGs$. By the Portmanteau Theorem, we deduce that for any $\varepsilon>0$,
\begin{eqnarray*}
\cL\left(f<-\varepsilon\right) & \le & \liminf_{n\to\infty}\cL_{G_n}\left(f<-\varepsilon\right)\\
& = & \liminf_{n\to\infty}\left\{\frac{1}{|V_n|}\mathrm{card}\{o\in V_n\colon f(G_n,o)<-\varepsilon\right\}\\
& \le & \liminf_{n\to\infty}\left\{\frac{2}{|V_n|}\mathrm{card}\{e\in E_n\colon \kappa_{G_n}(e)<-\varepsilon\right\}\\
& = &  \cL\left[\deg_G(o)\right]\liminf_{n\to\infty}\left\{\frac{1}{|E_n|}\mathrm{card}\{e\in E_n\colon \kappa_{G_n}(e)<-\varepsilon\right\},
\end{eqnarray*}
where the last inequality follows from the observation that $\frac{2|E_n|}{|V_n|}\to  \cL[\deg_G(o)]$, by the continuity and uniform integrability of $(G,o)\mapsto \deg_G(o)$. Sending $\varepsilon\to 0$ yields $\cL(f<0)=0$, by (\ref{assume:nonneg2}). To conclude, we simply apply Lemma \ref{lm:root} to the event $B=\{f\ge 0\}$. 
\end{proof}
The first crucial step in our proof consists in deducing from (\ref{kappaG}) that the entropy is zero under $\cL$. This is the content of the following theorem, which will be proved in Section \ref{sec:curvature}.
\begin{theorem}[Non-negative curvature implies zero-entropy]\label{th:curvature} The implication
\begin{eqnarray*}
\kappa(G)\ge 0 & \Longrightarrow & \rH(G)=0
\end{eqnarray*}
holds almost-surely under any stationary law $\cL\in\cP(\rGs)$ satisfying $\cL\left[\log\deg_G(o)\right]<\infty$.
\end{theorem}
In view of Remark \ref{rk:absolute}, this result also applies to any unimodular law $\cL\in\cP(\rGs)$ satisfying $\cL\left[\deg_G(o)\log\deg_G(o)\right]<\infty$, hence in particular to the limit $\cL$ in (\ref{lwc}). Combining this with Lemma \ref{lm:radius}, we immediately deduce that our local weak limit satisfies
\begin{eqnarray*}
\cL(\rho(G)=1) & = & 1.
\end{eqnarray*}
It turns out that this  simple condition suffices to guarantee (\ref{poorexpansion2}). This is the content of the following second result,  established in Section \ref{sec:spectrum} below, and which completes the proof of our main result.
\begin{theorem}[Zero-entropy implies poor spectral expansion]\label{th:spectrum} Let $G_n=(V_n,E_n), {n\ge 1}$ be finite graphs having local weak limit $\cL$, and suppose that $ \cL\left(\rho(G)=1\right)=1$. Then, for any $\rho<1$,
\begin{eqnarray*}
\liminf_{n\to\infty}\left\{\frac{1}{|V_n|}\,\mathrm{card}\left\{i\colon\lambda_i(G_n)> \rho\right\}\right\} & > & 0.
\end{eqnarray*}
\end{theorem}
In fact, a stronger statement about eigenvectors will be derived, as claimed in Remark \ref{rk:eigenvector}.
\subsection{Non-negative curvature implies zero entropy}
\label{sec:curvature}
Consider a connected graph $G$ and two vertices $x,y\in V_G$.
The proof of Theorem \ref{th:curvature} relies on the following intuitive idea: if $G$ has non-negative curvature and bounded degrees, then it takes time $O(\dd_G^2(x,y))$ for two random walks starting at  $x$ and $y$ to meet. This classical observation constitutes the very essence of the path coupling method of Bordewich and Dyer \cite{MR2316551}. It was later re-discovered and further developed by M\"unch \cite{2019arXiv190713514M}. We will here prove a refinement that does not require bounded degrees, see Corollary \ref{co:curvature} below.  Write $\cB_x,\cB_y$ for the balls of radius $1$ around $x$ and $y$, and recall that the Wassertein distance $\cW_1\left(P_G(x,\cdot),P_G(y,\cdot)\right)$ is defined as
\begin{eqnarray}
\label{def:cW1}
\cW_1\left(P_G(x,\cdot),P_G(y,\cdot)\right) & = & \inf_{\pi}\left\{\sum_{u\in\cB_x}\sum_{v\in \cB_y}\pi(u,v)\,\dd_G(u,v)\right\},
\end{eqnarray}
where the infimum runs over all probability distributions $\pi\in\cP(\cB_x\times\cB_y)$ with marginals $P_G(x,\cdot)$ and $P_G(y,\cdot)$. By compactness, the above infimum is actually achieved, and  the minimizers will be called \emph{optimal couplings}. As in  \cite{MR2316551,2019arXiv190713514M}, our first task consists in showing that an optimal coupling can always be chosen so as to assign a ``decent'' probability to the ``good'' set
\begin{eqnarray*}
\Gamma & := & \left\{(u,v)\in \cB_x\times \cB_y\colon \dd_G(u,v)<\dd_G(x,y)\right\}.
\end{eqnarray*}
The argument crucially uses the laziness of $P_G$ but is otherwise rather general.
\begin{lemma}[Good optimal couplings] \label{lm:coupling}If $x\ne y$, then there is an optimal coupling $\pi$ such that
\begin{eqnarray*}
\pi\left(\Gamma\right) & \ge & \frac{1}{2}\max\left\{\frac 1{\deg_G(x)},\frac 1{\deg_G(y)}\right\}.
\end{eqnarray*}
\end{lemma}
\begin{proof}  
By compactness, we can find an optimal coupling $\pi$ which, among all optimal couplings, maximizes   $\pi(\Gamma)$. Suppose for a contradiction that this ``doubly optimal'' coupling satisfies
\begin{eqnarray}
\label{absurd}
 \pi\left(\Gamma\right) & < & \frac{1}{2\deg_G(x)}.
\end{eqnarray}
The set ${A}:=\{u\in \cB_x\colon (u,y)\in \Gamma\}$ is not empty, since it contains the first vertex on a geodesic from $x$ to $y$. Thus, 
$
\pi(A\times \cB_y) \ge  1/(2\deg_G(x)).
$
In view of (\ref{absurd}), this forces  $\pi((A\times \cB_y)\setminus\Gamma)>0$, i.e.
\begin{eqnarray}
\label{exists:ab}\exists (x_0,y_0)\in (A\times \cB_y)\setminus\Gamma,\quad \pi(x_0,y_0)  & \ge &  \varepsilon,
\end{eqnarray}
for some $\varepsilon>0$. On the other hand, we have
$
\pi(A\times\{y\}) + \pi(A^c\times\{y\})  =  P_G(y,y) \ = \ \frac{1}{2}.
$
This forces $\pi(A^c\times\{y\}) > 0$, because $\pi(A\times\{y\})\le \pi(\Gamma)<\frac{1}{2}$. In other words, 
\begin{eqnarray}
\label{exists:z}
\exists x_1\in A^c,\quad \pi(x_1,y) & \ge & \varepsilon,
\end{eqnarray}
provided $\varepsilon>0$ is chosen small enough. We now use the vertices $x_0,y_0,x_1$ found at (\ref{exists:ab})-(\ref{exists:z}) to construct a new coupling $\wpi$ which contradicts the optimality of $\pi$. For all $(u,v)\in \cB_x\times \cB_y$, we set
\begin{eqnarray*}
\wpi(u,v) & := & \left\{
\begin{array}{ll}
\pi(u,v) & \textrm{ if }u\notin\{x_0,x_1\}\textrm{ and }b\notin\{y_0,y\};\\
\pi(u,v)-\varepsilon & \textrm{ if }(u,v)=(x_0,y_0)\textrm{ or }(u,v)=(x_1,y);\\
\pi(u,v)+\varepsilon & \textrm{ if }(u,v)=(x_0,y)\textrm{ or }(u,v)=(x_1,y_0).
\end{array}
\right.
\end{eqnarray*}
By construction, $\wpi$ is non-negative on $\cB_x\times \cB_y$ and has the same marginals as $\pi$. Thus, it is a coupling of $P_G(x,\cdot),P_G(y,\cdot)$. This coupling is moreover optimal, since
\begin{eqnarray*}
\sum_{u\in \cB_x}\sum_{v\in \cB_y}\dd_G(u,v)\left(\wpi(u,v)-\pi(u,v)\right) & = & \varepsilon\left(\dd_G(x_0,y)+\dd_G(x_1,y_0)-\dd_G(x_0,y_0)-\dd_G(x_1,y)\right)\\
& \le & \varepsilon\left(\dd_G(x,y)-1+\dd_G(x_1,y_0)-\dd_G(x,y)-\dd_G(x_1,y)\right)\\
& \le & 0,
\end{eqnarray*}
where the first inequality uses $x_0\in A$ and $(x_0,y_0)\notin\Gamma$, while the second uses the triangle inequality  $\dd_G(x_1,y_0)\le \dd_G(x_1,y)+\dd_G(y,y_0)$. Finally, since $\Gamma$ contains $(x_1,y)$ but not $(x_0,y_0),(x_1,y)$, we have
\begin{eqnarray*}
\wpi(\Gamma) & \ge & \pi(\Gamma)+\varepsilon,
\end{eqnarray*}
contradicting the definition of $\pi$. Thus, (\ref{absurd}) can not be true, and the claim follows by symmetry.
\end{proof}
We will also need the following technical lemma, which is of independent interest and quantifies the intuition that non-negative super-martingales that ``move a lot'' must ``quickly'' hit zero.  %Controls of this sort are classical, see e.g. \cite[Proposition 17.19]{MR3726904}.

\begin{lemma}[Non-negative super-martingales quickly hit  zero]\label{lm:martingale}Let  $\tau:=\inf\{t\ge 0\colon Z_t=0\}$ be the hitting time of zero by a  non-negative super-martingale $Z=(Z_t)_{t\ge 0}$. Suppose that  $Z_0=z$, and that all increments $(Z_{t+1}-Z_t)_{t\ge 0}$ are upper-bounded by a constant $K$. 
Then, 
\begin{eqnarray*}
\PP\left(\tau\ge t\right) & \le & z\left(\frac{2a+K-z}{a^2}\right)+\PP\left(\tau\ge t,\sum_{s=0}^{t-1}W_s< a^2\right),
\end{eqnarray*}
for all $t\in\dZ_+,a>0$, where $W_s=\EE\left[(Z_{s+1}-Z_s)^2|\cF_s\right]$ and   $(\cF_s)_{s\ge 0}$ is the underlying filtration.
\end{lemma}
\begin{proof}First note that the process $Z$ is trivially square-integrable, because $Z_t\in[0,z+Kt]$ for each $t\ge 0$. Now fix $t\ge 0$ and $a>0$, and consider the bounded stopping time 
\begin{eqnarray*}
\sigma & := & \inf\left\{s\ge 0\colon Z_s\ge a\right\}  \wedge t.
\end{eqnarray*}
Using the Optional Stopping Theorem, the non-negativity of $Z$ and the definition of $\sigma$, we have
\begin{eqnarray*}
z & \ge & \EE\left[Z_{\sigma\wedge\tau}\right] \\ &  \ge &  \EE\left[Z_{\sigma\wedge\tau}{\bf 1}_{(\sigma<\tau\wedge t)}\right]\\ & \ge & a \PP\left(\sigma<\tau\wedge t\right).
\end{eqnarray*}
On the other hand, observe that for all $s\ge 0$, we may rewrite $W_s$ as 
\begin{eqnarray*}
W_s & = & \EE\left[Z_{s+1}^2-Z_s^2|\cF_s\right]  +2Z_s\EE[Z_{s}-Z_{s+1}|\cF_s].
\end{eqnarray*}
Note that the second conditional expectation is non-negative by assumption. Moreover, we have $Z_s\le a$ on the event $\{\sigma>s\}$, which is in $\cF_s$. Thus, 
\begin{eqnarray*}
W_s{\bf 1}_{\sigma>s} & \le & \EE\left[\left(Z_{s+1}^2-Z_s^2\right){\bf 1}_{\sigma>s}|\cF_s\right]  +2a\EE\left[\left(Z_{s}-Z_{s+1}\right){\bf 1}_{\sigma>s}|\cF_s\right].
\end{eqnarray*}
Taking expectations and summing over all $s\ge 0$, we obtain
\begin{eqnarray*}
\EE\left[\sum_{s=0}^{\sigma-1}W_s\right]& \le & \EE\left[Z_{\sigma}^2\right] -2a\EE[Z_\sigma]-z^2+2az\\
& \le & (K+a-z)z,
\end{eqnarray*}
where the second inequality follows from the observations that $Z_\sigma\le K+a$ and $\EE[Z_\sigma]\le z$. Let us now use these two estimates to conclude. By union bound, we have
\begin{eqnarray*}
\PP\left(\tau\ge t\right) & \le & \PP\left(\sigma<\tau\wedge t\right)+\PP\left(\sigma\wedge\tau\ge t\right)\\
 & \le &  \PP\left(\sigma<\tau\wedge t\right)+\PP\left(\tau\ge t,\sum_{s=0}^{\sigma-1}W_s\ge \sum_{s=0}^{t-1}W_s\right)\\
 & \le &  \PP\left(\sigma<\tau\wedge t\right)+ \PP\left(\sum_{s=0}^{\sigma-1}W_s\ge a^2\right)+\PP\left(\tau\ge t,\sum_{s=0}^{t-1}W_s<a^2\right)\\
 & \le & \frac{z}{a}+\frac{(K+a-z)z}{a^2}+ \PP\left(\tau\ge t,\sum_{s=0}^{t-1}W_s<a^2\right).
\end{eqnarray*}
This is exactly the claimed bound.
\end{proof}

Combining these two lemmas, we may now deduce the following estimate, which exploits non-negative curvature to control the action of $P_G$ on the variations of bounded observables.
\begin{proposition}[Variational estimate via non-negative curvature]\label{pr:harmo}Let  $G$ be a connected graph with $\kappa(G)\ge 0$.  Then, for any  $f\colon V_G\to [-1,1]$, any vertices $x,y\in V_G$, and any $a>0,t\in\dZ_+$, 
\begin{eqnarray*}
{|P^t_Gf(x)-P^t_Gf(y)|} & \le & \frac{8\dd_G(x,y)}{a}  + 2\PP\left(\sum_{s=0}^{t-1}\frac{1}{\deg_G(X_s)}<2a^2\right),
\end{eqnarray*}
where $X$ denotes a lazy random walk on $G$ starting from $x$.
\end{proposition}

\begin{proof}Let $(X,Y)$ be the Markov chain on $V_G\times V_G$ which, from any state $(x,y)\in V_G\times V_G$, draws the next state according to the ``good'' optimal coupling  of $P_G(x,\cdot),P_G(y,\cdot)$ described in Lemma \ref{lm:coupling}. We use the standard notations $\PP_{(x,y)}(\cdot),\EE_{(x,y)}[\cdot]$ to specify the choice of the initial state. Since the two coordinates $X,Y$ are marginally distributed as lazy random walks on $G$, we have
\begin{eqnarray*}
\left|P^t_Gf(x)-P^t_Gf(y)\right| 
& = & \left|\EE_{x,y}\left[f(X_t)\right]-\EE_{x,y}\left[f(Y_t)\right]\right|\\
& \le & \EE_{x,y}\left[\left|f(X_t)-f(Y_t)\right|\right]\\
& \le & 2\PP_{x,y}\left(X_t\ne Y_t\right)\\
& \le & 2\PP_{x,y}\left(\tau> t\right),
\end{eqnarray*}
where $\tau=\inf\{t\ge 0\colon X_t=Y_t\}$ denotes the meeting time of the two walkers. Note that $\tau$ is also the hitting time of zero by the non-negative process $Z=(Z_t)_{t\ge 0}$ defined as follows: 
\begin{eqnarray*}
\forall t\ge 0,\quad Z_t & := & \dd_G(X_t,Y_t).
\end{eqnarray*}
We claim that $Z$ is a super-martingale w.r.t. the natural filtration  $(\cF_t)_{t\ge 0}$ associated with $(X,Y)$. Indeed, by the Markov property and the optimality of the chosen couplings, this claim reduces to 
\begin{eqnarray*}
\cW_1\left(P_G(x,\cdot),P_G(y,\cdot)\right) & \le &  \dd_G(x,y),
\end{eqnarray*}
for all $x,y\in V_G$. But this inequality readily follows from the assumption $\kappa_G(x,y)\ge 0$ in the case $\{x,y\}\in E_G$, and it then automatically extends to all $x,y\in V_G$ by the triangle inequality of  $\cW_1(\cdot,\cdot)$ (see, e.g.,  \cite{MR1964483}). On the other hand, Lemma \ref{lm:coupling} ensures that on the event $\{\tau> t\}$, 
\begin{eqnarray*}
\EE_{x,y}\left[(Z_{t+1}-Z_t)^2|\cF_t\right] & \ge & \frac{1}{2\deg_G(X_t)}.
\end{eqnarray*}
Finally, note that the distance between the two walkers can not increase by more than $2$ at each step. Thus, we may invoke Lemma \ref{lm:martingale} to conclude that
\begin{eqnarray*}
\PP_{x,y}\left(\tau\ge t\right) & \le & 2\dd_G(x,y)\left(\frac{a+1}{a^2}\right)+\PP_{x,y}\left(\sum_{s=0}^{t-1}\frac{1}{\deg_G(X_s)}< 2a^2\right)\\
& \le & \frac{4\dd_G(x,y)}{a}+\PP_{x,y}\left(\sum_{s=0}^{t-1}\frac{1}{\deg_G(X_s)}< 2a^2\right),
\end{eqnarray*}
where the second line follows from the first if  $a\ge 1$, and is trivial otherwise.
\end{proof}
In particular, this applies to any bounded harmonic function $f$, after a trivial normalization. Since  $P^t_Gf=f$ for all $t\ge 0$, we may send $t\to\infty$ and then $a\to\infty$ in the resulting estimate to obtain the following key result, which ensures that  non-negatively curved graphs satisfy the Liouville property, provided they have a ``decent proportion'' of vertices with ``reasonable'' degree. 
\begin{corollary}[Liouville property and non-negative curvature]\label{co:curvature}Let $G$ be a connected graph with $\kappa(G)\ge 0$. Fix $o\in V_G$ and suppose that the simple random walk $X$ on $G$ starting from $o$ satisfies
\begin{eqnarray}
\label{degrees}
\PP\left(\sum_{t=0}^\infty\frac{1}{\deg_G(X_t)} = \infty\right) & = & 1.
\end{eqnarray}
 Then, $G$ has the Liouville property.
\end{corollary}
A simple situation where the above condition trivially holds is that where $G$ has bounded degrees. In that case,   the Liouville property was recently established by Jost, Münch, and  Rose  \cite{jost2019liouville}. Our relaxation allows for arbitrary large degrees, as long as the random walk can avoid them from times to times. This is the case under any stationary law by Birkhoff's Ergodic Theorem, allowing us to prove Theorem \ref{th:curvature}. 
\begin{proof}[Proof of Theorem \ref{th:curvature}]Let $(G,o)$ have law $\cL$ and, conditionally on $(G,o)$, let $X$ be a lazy random walk starting from the root. Then the process $Z=(Z_t)_{t\ge 0}$ defined by 
\begin{eqnarray*}
\forall t\ge 0,\quad Z_t & := & \frac{1}{\deg_G(X_t)}
\end{eqnarray*}
is stationary, in the usual sense that its law is invariant under the shift $(z_t)_{t\ge 0}\mapsto (z_{t+1})_{t\ge 0}$ on $[0,1]^{\dZ_+}$. Thus, Birkhoff's Ergodic Theorem (see, e.g. \cite[Theorem 14.43]{MR3616205}) ensures that 
\begin{eqnarray*}
\frac{1}{t}\sum_{s=0}^{t-1}Z_s & \xrightarrow[t\to\infty]{} &  \EE[Z_1|\rI],
\end{eqnarray*}
almost-surely, where $\rI$ is the invariant $\sigma-$algebra. Since $Z_1$ is almost-surely positive, we deduce 
\begin{eqnarray*}
\sum_{s=0}^{\infty}Z_s & = &  \infty,
\end{eqnarray*}
almost-surely. In other words, the random graph $(G,o)$ satisfies (\ref{degrees}) almost-surely. By the above corollary, this implies that $G$ has the Liouville property almost-surely on the event $\{\kappa(G)\ge 0\}$. By Theorem \ref{th:entropy}, we conclude that $\rH(G)=0$ almost-surely on the same event.
\end{proof}

\subsection{Zero entropy implies poor spectral expansion}
\label{sec:spectrum}
This final section is devoted to proving Theorem \ref{th:spectrum}, which relates the eigenvalues of finite graphs  to the spectral radius of their local weak limits. If $G$ is a finite graph, the $N=|V_G|$ eigenvalues $\lambda_1(G)\ge\ldots\ge\lambda_{N}(G)$ of its transition matrix $P_G$ can be conveniently encoded into a probability measure $\mu_G\in\cP([0,1])$, called the \emph{empirical eigenvalue distribution} of the matrix $P_G$:
\begin{eqnarray*}
\mu_G & := & \frac{1}{N}\sum_{i=1}^{N}\delta_{\lambda_i(G)}.
\end{eqnarray*}
It turns out that the large-size asymptotics of this fundamental object can be understood directly at the level of local weak limits. When $P_G$ is replaced with the more standard adjacency matrix, this classical observation is the starting point of a rich and well-established theory, see the comprehensive introductory survey \cite{MR3792625} by Bordenave, and the references therein. 
\paragraph{Local spectral measures.} 
The transition kernel $P_G$ of a graph $G$ can be viewed  as a linear operator acting  via (\ref{def:P}) on  the Hilbert space 
\begin{eqnarray*}
\ell^2(G) & := & \left\{f\in \dC^{V_G}\colon 
 \sum_{o\in V_G}\deg_G(o)|f(o)|^2 < \infty\right\}, 
\end{eqnarray*}
with inner product $\langle f,g\rangle   =   \sum_{o\in V_G}\deg_G(o)\overline{f(o)}g(o)$. 
The stochasticity, laziness and reversibility 
\begin{eqnarray*}
\sum_{y\in V_G}P_G(x,y)=1, \qquad P_G(x,x)\ge 1/2, \qquad \deg_G(x)P_G(x,y) = \deg_G(y)P_G(y,x),
\end{eqnarray*}
easily (and classically) imply that $P_G$ is a positive contraction on $\ell^2(G)$, i.e.
\begin{eqnarray*}
\forall f\in\ell^2(G),\qquad 0 \ \le \ \langle f,P_Gf\rangle\ \le \langle \ f,f\rangle.
\end{eqnarray*}
In particular, for each $o\in V_G$, the spectral theorem for self-adjoint operators ensures the existence of a \emph{local spectral measure} $\mu_{(G,o)}\in\cP([0,1])$, characterized by the moment identity
\begin{eqnarray}
\label{def:moments}
\forall t\ge 0,\quad  \int_{0}^1 \lambda^t\mu_{(G,o)}(\dd\lambda) & = & P^t_G(o,o).
\end{eqnarray}
As we will now see,  $\mu_{(G,o)}$ can be interpreted as the local contribution of  $o$ to the spectrum of $P_G$.  Local spectral measures are a  powerful tool to investigate the mixing properties of graphs, see  \cite{MR3892273}.

\paragraph{The finite case.} When $G$ is finite with $N$ vertices, there is an orthonormal basis $(\phi_1,\ldots,\phi_{N})$  of $\ell^2(G)$ consisting of eigenvectors of $P_G$ with eigenvalues $\lambda_1(G),\ldots,\lambda_{N}(G)$, and we easily find
\begin{eqnarray}
\label{spectralmeasure}
\mu_{(G,o)} & = & \sum_{i=1}^{N}\deg_G(o)|\phi_i(o)|^2\delta_{\lambda_i(G)}.
\end{eqnarray}
Thus, the local spectral measure $\mu_{(G,o)}$ is a mixture of Dirac masses located at the various eigenvalues of $P_G$, and weighted by the squared amplitudes of the corresponding eigenvectors at  $o$. Moreover, thanks to the orthonormality of $(\phi_1,\ldots,\phi_{N})$, the identity (\ref{spectralmeasure}) readily implies  
\begin{eqnarray}
\label{localglobal}
\mu_G & = & \frac{1}{|V_G|}\sum_{o\in V_G}\mu_{(G,o)}.
\end{eqnarray}
In other words, the empirical eigenvalue distribution of a finite graph $G$ coincides with the spatial average of its local spectral measures. 

\paragraph{Spectral continuity.} In light of  (\ref{def:lwc}), it is tempting to pass to the limit in the formula (\ref{localglobal}) along a convergent sequence of finite graphs $(G_n)_{n\ge 1}$. This is made rigorous by the following continuity principle. As usual, $\cP([0,1])$ is here equipped with the topology of weak convergence.
\begin{lemma}[Spectral continuity]\label{lm:continuity}The map $(G,o) \mapsto \mu_{(G,o)}$ is continuous on $\rGs$. 
In particular, if a sequence of graphs $(G_n)_{n\ge 1}$ admits a local weak limit $\cL$, then
\begin{eqnarray*}
\mu_{G_n}(\dd\lambda) & \xrightarrow[n\to\infty]{} & \mu_\cL(\dd\lambda)\ :=\  \cL\left[\mu_{(G,o)}(\dd\lambda)\right].
\end{eqnarray*}
\end{lemma}
\begin{proof}For each fixed $t\ge 0$, the observable $(G,o)\mapsto P^t_G(o,o)$ is clearly $t-$local, hence continuous. In particular, via  the identity (\ref{def:moments}), the convergence $(G_n,o_n)\to (G,o)$ in $\rGs$ implies
\begin{eqnarray}
\label{cv:moments}
\forall t\ge 0,\quad \int_{0}^1 \lambda^t\,\mu_{(G_n,o_n)}(\dd\lambda)  & \xrightarrow[n\to\infty]{} &  \int_{0}^1 \lambda^t\,\mu_{(G,o)}(\dd\lambda).
 \end{eqnarray} 
Since convergence in $\cP([0,1])$ is equivalent to the convergence of moments, we conclude that
 $\mu_{(G_n,o_n)}\xrightarrow[n\to\infty]{}\mu_{(G,o)}$, and the continuity is proved. Similarly, the second claim  is obtained by applying (\ref{def:lwc}) to the $t-$local observable $f\colon (G,o)\mapsto P^t_G(o,o)$, for each $t\ge 1$.
\end{proof}

\begin{corollary}[Unit spectral radius implies poor spectral expansion]\label{co:spectrum}Let $G_n=(V_n,E_n), {n\ge 1}$ be  finite graphs having a local weak limit $\cL$ such that $\cL(\rho(G)=1)=1$. Then, for any $0\le \rho<1$,  
\begin{eqnarray}
\label{accumulation}
\liminf_{n\to\infty}\, \mu_{G_n}\left([\rho,1]\right) & > & 0.
\end{eqnarray}
Moreover, we have the refinement
\begin{eqnarray}
\label{eigenvectors}
\sup_{n\ge 1}\,\frac{\left|\left\{x\in V_n\colon \mu_{(G_n,x)}([\rho,1])\le \varepsilon\right\}\right|}{|V_n|}& \xrightarrow[\varepsilon\to0]{} & 0. 
\end{eqnarray}
\end{corollary}
\begin{proof}Fix $0\le \rho<1$. By the second part of Lemma \ref{lm:continuity} and the Portmanteau Theorem, we have
\begin{eqnarray}
\label{radiusliminf}
\liminf_{n\to\infty}\mu_{G_n}([\rho,1]) & \ge & \cL\left[\mu_{(G,o)}((\rho,1])\right].
\end{eqnarray}
On the other hand, comparing  (\ref{def:moments}) with the definition of the spectral radius, we see that $\rho(G)$ is exactly  the supremum of the support of $\mu_{(G,o)}$, for any   $(G,o)\in\rGs$. In other words,  
\begin{eqnarray*}
 \mu_{(G,o)}((\rho,1])>0 & \Longleftrightarrow & \rho(G)>\rho.
\end{eqnarray*}
In particular, since $\cL(\rho(G)=1)=1$, the right-hand side of (\ref{radiusliminf}) is positive, as desired. To prove the second claim, note that the continuity of $(G,o)\mapsto\mu_{(G,o)}$ implies that the event $F_\varepsilon=\left\{\mu_{(G,o)}([\rho,1])\le \varepsilon\right\}$ is closed in $\rGs$. Consequently, the convergence $G_n\to\cL$ implies
\begin{eqnarray*}
\limsup_{n\to\infty}\cL_{G_n}(F_\varepsilon) & \le & \cL(F_\varepsilon),
\end{eqnarray*}
and the right-hand side tends  to $\cL(F_0)\le \cL(\rho(G)\le\rho)=0$ as $\varepsilon\to 0$.  
The limsup can then be replaced with a sup, since   for each $n\ge 1$, $\cL_{G_n}(F_\varepsilon)$   decreases monotonically to $0$ with $\varepsilon$. 
\end{proof}
\begin{remark}[Corollary \ref{co:spectrum} vs Theorem \ref{th:spectrum}]\label{rk:final}
The statement (\ref{accumulation})  asserts that a macroscopic proportion of eigenvalues of $G_n$  accumulate in $[\rho,1]$, which is exactly the conclusion of Theorem \ref{th:spectrum}. The refinement (\ref{eigenvectors}), on the other hand, constitutes a rigorous formalization of the ``delocalization'' announced in Remark \ref{rk:eigenvector}. To see this, recall that for any graph $G$ with $N$ vertices, we have by (\ref{spectralmeasure}),
\begin{eqnarray*}
\mu_{(G,x)}([\rho,1]) & = & \sum_{i=1}^{N}\deg_G(x)|\phi_i(x)|^2{\bf 1}_{\lambda_i(G)\ge \rho}.
\end{eqnarray*}
In words, the number $\mu_{(G,x)}([\rho,1])\in[0,1]$  measures the cumulative squared amplitude at   $x$ of all the basis eigenvectors corresponding to ``bad'' eigenvalues (those in $[\rho,1]$). In particular, the set $\{x\in V_G\colon \mu_{(G,x)}([\rho,1])\le \varepsilon\}$ represents the region where these ``bad'' eigenvectors have a small cumulative squared amplitude. The statement (\ref{eigenvectors}) asserts that the relative size of this region can be made arbitrarily small by choosing $\varepsilon$ small, uniformly in $n$. Thus, bad eigenvectors have their cumulative mass ``spread out'' across most vertices. 
\end{remark}

\bibliographystyle{plain}
\bibliography{draft}

\begin{thebibliography}{10}

\bibitem{MR2354165}
David Aldous and Russell Lyons.
\newblock Processes on unimodular random networks.
\newblock {\em Electron. J. Probab.}, 12:no. 54, 1454--1508, 2007.

\bibitem{MR2023650}
David Aldous and J.~Michael Steele.
\newblock The objective method: probabilistic combinatorial optimization and
  local weak convergence.
\newblock In {\em Probability on discrete structures}, volume 110 of {\em
  Encyclopaedia Math. Sci.}, pages 1--72. Springer, Berlin, 2004.

\bibitem{MR3449319}
Venkat Anantharam and Justin Salez.
\newblock The densest subgraph problem in sparse random graphs.
\newblock {\em Ann. Appl. Probab.}, 26(1):305--327, 2016.

\bibitem{MR0507229}
A.~Avez.
\newblock Harmonic functions on groups.
\newblock In {\em Differential geometry and relativity}, pages 27--32.
  Mathematical Phys. and Appl. Math., Vol. 3. 1976.

\bibitem{MR324741}
Andr\'{e} Avez.
\newblock Entropie des groupes de type fini.
\newblock {\em C. R. Acad. Sci. Paris S\'{e}r. A-B}, 275:A1363--A1366, 1972.

\bibitem{MR353405}
Andr\'{e} Avez.
\newblock Th\'{e}or\`eme de {C}hoquet-{D}eny pour les groupes \`a croissance
  non exponentielle.
\newblock {\em C. R. Acad. Sci. Paris S\'{e}r. A}, 279:25--28, 1974.

\bibitem{MR889476}
D.~Bakry and Michel \'{E}mery.
\newblock Diffusions hypercontractives.
\newblock In {\em S\'{e}minaire de probabilit\'{e}s, {XIX}, 1983/84}, volume
  1123 of {\em Lecture Notes in Math.}, pages 177--206. Springer, Berlin, 1985.

\bibitem{MR941980}
Dominique Bakry.
\newblock \'{E}tude des transformations de {R}iesz dans les vari\'{e}t\'{e}s
  riemanniennes \`a courbure de {R}icci minor\'{e}e.
\newblock In {\em S\'{e}minaire de {P}robabilit\'{e}s, {XXI}}, volume 1247 of
  {\em Lecture Notes in Math.}, pages 137--172. Springer, Berlin, 1987.

\bibitem{MR3155209}
Dominique Bakry, Ivan Gentil, and Michel Ledoux.
\newblock {\em Analysis and geometry of {M}arkov diffusion operators}, volume
  348 of {\em Grundlehren der Mathematischen Wissenschaften [Fundamental
  Principles of Mathematical Sciences]}.
\newblock Springer, Cham, 2014.

\bibitem{MR2994841}
Itai Benjamini and Nicolas Curien.
\newblock Ergodic theory on stationary random graphs.
\newblock {\em Electron. J. Probab.}, 17:no. 93, 20, 2012.

\bibitem{MR3395463}
Itai Benjamini, Hugo Duminil-Copin, Gady Kozma, and Ariel Yadin.
\newblock Disorder, entropy and harmonic functions.
\newblock {\em Ann. Probab.}, 43(5):2332--2373, 2015.

\bibitem{MR3316916}
Itai Benjamini, Russell Lyons, and Oded Schramm.
\newblock Unimodular random trees.
\newblock {\em Ergodic Theory Dynam. Systems}, 35(2):359--373, 2015.

\bibitem{MR3813982}
Itai Benjamini, Elliot Paquette, and Joshua Pfeffer.
\newblock Anchored expansion, speed and the {P}oisson-{V}oronoi tessellation in
  symmetric spaces.
\newblock {\em Ann. Probab.}, 46(4):1917--1956, 2018.

\bibitem{MR1873300}
Itai Benjamini and Oded Schramm.
\newblock Recurrence of distributional limits of finite planar graphs.
\newblock {\em Electron. J. Probab.}, 6:no. 23, 13, 2001.

\bibitem{MR3792625}
Charles Bordenave.
\newblock Spectrum of random graphs.
\newblock In {\em Advanced topics in random matrices}, volume~53 of {\em Panor.
  Synth\`eses}, pages 91--150. Soc. Math. France, Paris, 2017.

\bibitem{MR2789584}
Charles Bordenave, Marc Lelarge, and Justin Salez.
\newblock The rank of diluted random graphs.
\newblock {\em Ann. Probab.}, 39(3):1097--1121, 2011.

\bibitem{MR3101844}
Charles Bordenave, Marc Lelarge, and Justin Salez.
\newblock Matchings on infinite graphs.
\newblock {\em Probab. Theory Related Fields}, 157(1-2):183--208, 2013.

\bibitem{MR2316551}
Magnus Bordewich and Martin Dyer.
\newblock Path coupling without contraction.
\newblock {\em J. Discrete Algorithms}, 5(2):280--292, 2007.

\bibitem{MR3815539}
D.~P. Bourne, D.~Cushing, S.~Liu, F.~M\"{u}nch, and N.~Peyerimhoff.
\newblock Ollivier-{R}icci idleness functions of graphs.
\newblock {\em SIAM J. Discrete Math.}, 32(2):1408--1424, 2018.

\bibitem{MR3546392}
Mat\'{\i}as Carrasco~Piaggio and Pablo Lessa.
\newblock Equivalence of zero entropy and the {L}iouville property for
  stationary random graphs.
\newblock {\em Electron. J. Probab.}, 21:Paper No. 55, 24, 2016.

\bibitem{MR4096132}
D.~Cushing, S.~Kamtue, J.~Koolen, S.~Liu, F.~M\"{u}nch, and N.~Peyerimhoff.
\newblock Rigidity of the {B}onnet-{M}yers inequality for graphs with respect
  to {O}llivier {R}icci curvature.
\newblock {\em Adv. Math.}, 369:107188, 53, 2020.

\bibitem{doi:10.1080/10586458.2019.1660740}
David Cushing, Riikka Kangaslampi, Valtteri Lipiäinen, Shiping Liu, and
  George~W. Stagg.
\newblock The graph curvature calculator and the curvatures of cubic graphs.
\newblock {\em Experimental Mathematics}, 0(0):1--13, 2019.

\bibitem{MR4045968}
David Cushing, Shiping Liu, and Norbert Peyerimhoff.
\newblock Bakry-\'{E}mery curvature functions on graphs.
\newblock {\em Canad. J. Math.}, 72(1):89--143, 2020.

\bibitem{MR3726607}
Ronen Eldan, James~R. Lee, and Joseph Lehec.
\newblock Transport-entropy inequalities and curvature in discrete-space
  {M}arkov chains.
\newblock In {\em A journey through discrete mathematics}, pages 391--406.
  Springer, Cham, 2017.

\bibitem{MR2776719}
G\'{a}bor Elek.
\newblock On the limit of large girth graph sequences.
\newblock {\em Combinatorica}, 30(5):553--563, 2010.

\bibitem{MR3706773}
Max Fathi and Yan Shu.
\newblock Curvature and transport inequalities for {M}arkov chains in discrete
  spaces.
\newblock {\em Bernoulli}, 24(1):672--698, 2018.

\bibitem{MR3910593}
Bobo Hua.
\newblock Liouville theorem for bounded harmonic functions on manifolds and
  graphs satisfying non-negative curvature dimension condition.
\newblock {\em Calc. Var. Partial Differential Equations}, 58(2):Paper No. 42,
  8, 2019.

\bibitem{MR3726907}
J\"{u}rgen Jost.
\newblock {\em Riemannian geometry and geometric analysis}.
\newblock Universitext. Springer, Cham, seventh edition, 2017.

\bibitem{jost2019liouville}
Jürgen Jost, Florentin Münch, and Christian Rose.
\newblock Liouville property and non-negative ollivier curvature on graphs,
  2019.

\bibitem{MR2683634}
Ald\'{e}ric Joulin and Yann Ollivier.
\newblock Curvature, concentration and error estimates for {M}arkov chain
  {M}onte {C}arlo.
\newblock {\em Ann. Probab.}, 38(6):2418--2442, 2010.

\bibitem{MR704539}
V.~A. Ka\u{\i}manovich and A.~M. Vershik.
\newblock Random walks on discrete groups: boundary and entropy.
\newblock {\em Ann. Probab.}, 11(3):457--490, 1983.

\bibitem{MR4149342}
Mark Kempton, Gabor Lippner, and Florentin M\"{u}nch.
\newblock Large scale {R}icci curvature on graphs.
\newblock {\em Calc. Var. Partial Differential Equations}, 59(5):Paper No. 166,
  17, 2020.

\bibitem{MR3492631}
Bo'az Klartag, Gady Kozma, Peter Ralli, and Prasad Tetali.
\newblock Discrete curvature and abelian groups.
\newblock {\em Canad. J. Math.}, 68(3):655--674, 2016.

\bibitem{MR2872958}
Yong Lin, Linyuan Lu, and Shing-Tung Yau.
\newblock Ricci curvature of graphs.
\newblock {\em Tohoku Math. J. (2)}, 63(4):605--627, 2011.

\bibitem{MR2160416}
Russell Lyons.
\newblock Asymptotic enumeration of spanning trees.
\newblock {\em Combin. Probab. Comput.}, 14(4):491--522, 2005.

\bibitem{MR3892273}
Russell Lyons and Shayan Oveis~Gharan.
\newblock Sharp bounds on random walk eigenvalues via spectral embedding.
\newblock {\em Int. Math. Res. Not. IMRN}, (24):7555--7605, 2018.

\bibitem{MR1336708}
Russell Lyons, Robin Pemantle, and Yuval Peres.
\newblock Ergodic theory on {G}alton-{W}atson trees: speed of random walk and
  dimension of harmonic measure.
\newblock {\em Ergodic Theory Dynam. Systems}, 15(3):593--619, 1995.

\bibitem{MR3616205}
Russell Lyons and Yuval Peres.
\newblock {\em Probability on trees and networks}, volume~42 of {\em Cambridge
  Series in Statistical and Probabilistic Mathematics}.
\newblock Cambridge University Press, New York, 2016.

\bibitem{2019arXiv190910242M}
Florentin {M{\"u}nch}.
\newblock {Li-Yau inequality under $CD(0,n)$ on graphs}.
\newblock {\em arXiv e-prints}, page arXiv:1909.10242, September 2019.

\bibitem{2019arXiv190713514M}
Florentin {M{\"u}nch}.
\newblock {Non-negative Ollivier curvature on graphs, reverse Poincar{\'e}
  inequality, Buser inequality, Liouville property, Harnack inequality and
  eigenvalue estimates}.
\newblock {\em arXiv e-prints}, page arXiv:1907.13514, July 2019.

\bibitem{MR3998765}
Florentin M\"{u}nch and Rados\l aw~K. Wojciechowski.
\newblock Ollivier {R}icci curvature for general graph {L}aplacians: heat
  equation, {L}aplacian comparison, non-explosion and diameter bounds.
\newblock {\em Adv. Math.}, 356:106759, 45, 2019.

\bibitem{MR2371483}
Yann Ollivier.
\newblock Ricci curvature of metric spaces.
\newblock {\em C. R. Math. Acad. Sci. Paris}, 345(11):643--646, 2007.

\bibitem{MR2484937}
Yann Ollivier.
\newblock Ricci curvature of {M}arkov chains on metric spaces.
\newblock {\em J. Funct. Anal.}, 256(3):810--864, 2009.

\bibitem{MR2648269}
Yann Ollivier.
\newblock A survey of {R}icci curvature for metric spaces and {M}arkov chains.
\newblock In {\em Probabilistic approach to geometry}, volume~57 of {\em Adv.
  Stud. Pure Math.}, pages 343--381. Math. Soc. Japan, Tokyo, 2010.

\bibitem{MR1964483}
C\'{e}dric Villani.
\newblock {\em Topics in optimal transportation}, volume~58 of {\em Graduate
  Studies in Mathematics}.
\newblock American Mathematical Society, Providence, RI, 2003.

\end{thebibliography}
\end{document}